\documentclass{article}
\usepackage[utf8]{inputenc}
\usepackage{xcolor,ulem}
\usepackage{authblk}

\usepackage{amsfonts,amsmath}
\usepackage{amsthm}
\usepackage[polish,main=english]{babel}
\usepackage{array}
\usepackage{hyperref}
\usepackage{bbm} 
\usepackage{graphicx}
\usepackage{mathtools}
\usepackage{algorithm}
\usepackage{algorithmic}
\usepackage{enumitem}
\usepackage{caption}
\usepackage{subcaption}
\usepackage{cancel}

\DeclarePairedDelimiterX{\inp}[2]{\langle}{\rangle}{#1, #2}

\newtheorem{theorem}{Theorem}[section]
\newtheorem{corollary}{Corollary}[theorem]
\newtheorem{remark}{Remark}[theorem]
\newtheorem{fact}{Fact}[theorem]
\newtheorem{example}{Example}[theorem]

\newtheorem{definition}{Definition}[theorem]
\newtheorem{proposition}{Proposition}[theorem]
\newtheorem{lemma}[theorem]{Lemma}

\usepackage{todonotes}
\usepackage{marginnote}

\newcommand{\lefttodo}[2][]{{%
		\let\marginpar\marginnote
		\reversemarginpar
		\renewcommand{\baselinestretch}{0.8}%
		\todo[#1]{#2}}}

\newcommand{\ols}{{\rm OLS}}
\newcommand{\slope}{{\rm SLOPE}}

\newcommand{\bbeta}{\mbox{\boldmath $\beta$}}
\newcommand{\bY}{\mbox{\boldmath $Y$}}
\newcommand{\bb}{\mbox{\boldmath $b$}}
\newcommand{\bX}{\mbox{\boldmath $X$}}

\newcommand{\bLambda}{\mbox{\boldmath $\Lambda$}}
\newcommand{\beps}{\mbox{\boldmath $\varepsilon$}}
\newcommand{\bpi}{\mbox{\boldmath $\pi$}}
\newcommand{\bI}{\mbox{\boldmath $I$}}
\newcommand{\bnull}{\mbox{\boldmath $0$}}

\newcommand{\bC}{\mbox{\boldmath $C$}}

\newcommand{\bU}{\mbox{\boldmath $U$}}
\newcommand{\bS}{\mbox{\boldmath $S$}}
\newcommand{\bz}{\mbox{\boldmath $z$}}

\newcommand{\Then}{\Rightarrow}

\newcommand{\sign}{{\rm sign}}

\newcommand{\pattern}{{\mbox{\boldmath{$patt$}}}}

\newcommand{\betaLASSO}{\hat\beta^{\rm LASSO}}

\newcommand{\betaSLOPE}{\hat\beta^{\rm SLOPE}}

\renewcommand{\P}{\mathbb{P}}
\renewcommand{\epsilon}{\varepsilon}
\newcommand{\eps}{\varepsilon}
\newcommand{\lt}{\left}
\newcommand{\rt}{\right}

\newcommand{\R}{\mathbb{R}}
\newcommand{\Z}{\mathbb{Z}}

\begin{document}
	\title{Pattern recovery and signal denoising by SLOPE when the design matrix is orthogonal
		\footnote{The authors want to thank Ma\l{}gorzata Bogdan, Patrick Tardivel, Wojciech Rejchel and Tomasz \.Zak for references, comments and discussion. The authors are thankful to anonymous referee for simplifying the proof of Lemma~\ref{lemat 1}. First Author was supported by a French Government Scholarship. 
			Research of the first and the second author was supported by Centre Henri Lebesgue, program ANR-11-LABX-0020-0.
	}}
	
	\author[1,2]{Tomasz Skalski \thanks{corresponding author, email: Tomasz.Skalski@pwr.edu.pl}}
	\author[2]{Piotr Graczyk \thanks{email: graczyk@univ-angers.fr}}
	\author[3]{Bartosz Ko\l{}odziejek \thanks{email: b.kolodziejek@mini.pw.edu.pl}}
	\author[1]{Maciej Wilczy\'nski \thanks{email: Maciej.Wilczynski@pwr.edu.pl}}
	\affil[1]{
		Faculty of Pure and Applied Mathematics, Wroc\l{}aw University of Science and Technology, Wybrze\.{z}e Wyspia\'{n}skiego 27, 50-370 Wroc\l{}aw, Poland}
	\affil[2]{
		Laboratoire de Math\'ematiques LAREMA, Universit\'e d’Angers, France}
	\affil[3]{
		Faculty of Mathematics and Information Science, Warsaw University of Technology, ul. Koszykowa 75, 00-662 Warsaw, Poland}
	\date{\today}

	\maketitle
	
	\begin{abstract}
		Sorted $\ell_1$ Penalized Estimator (SLOPE) is a relatively new convex regularization method for fitting high-dimensional regression models.
		SLOPE allows to reduce the model dimension by shrinking some estimates of the regression coefficients completely to zero or by equating the absolute values of some nonzero estimates of these coefficients.
		This allows to identify situations where some of~true regression coefficients are equal. In this article we will introduce the SLOPE pattern, i.e., the set of relations between the true regression coefficients, which can be identified by SLOPE.
		We will also present new results on the strong consistency of SLOPE estimators and on~the~strong consistency of pattern recovery by~SLOPE when the design matrix is orthogonal and illustrate advantages of~the~SLOPE clustering in the context of high frequency signal denoising.
	\end{abstract}

	\newpage
	\section{Introduction}
	\subsection{Introduction and motivations}
	The Linear Multiple Regression concerns the model $\bY=\bX\bbeta+\beps$, where $\bY\in\R^n$ is~an~output vector, $\bX\in\R^{n\times p}$ is a~fixed design matrix, $\bbeta\in\R^p$ is an unknown vector of predictors and~$\beps\in\R^n$ is a~noise vector. The primary goal is to estimate $\bbeta$. In the low-dimensional setting, i.e., when the number of predictors $p$ is not larger than the number of explanatory variables $n$ and $\bX$ is of full rank, the ordinary least squares estimator $\hat\bbeta^{OLS}$ has an exact formula $\hat\bbeta^{OLS}=(\bX'\bX)^{-1}\bX'\bY$. 
	For practical reasons there is~an~urge to~avoid the~high-dimensionality curse, therefore we want the estimate to be sparse, i.e., to~be~descriptible by~a~smaller number of parameters. Several solutions were proposed to~deal~with such problem. One of them, the Least Absolute Shrinkage and Selection Operator (LASSO~\cite{chen, tibshirani1996regression}) involves penalizing the residual sum of~squares $\|\bY-\bX\hat\bbeta\|_{2}^{2}$ with an $\ell_1$ norm of $\hat\bbeta$ multiplied by~a~tuning parameter $\lambda$:
	\begin{equation*}
		\hat\bbeta^{LASSO} := \underset{{\mathbf b}\in\R^p}{\arg\min} \left[\frac12\|\bY-\bX\bb\|^{2}_{2} + \lambda \|\bb\|_1\right].
	\end{equation*}
	The LASSO estimator is not unbiased, but is a shrinkage estimator which shrinks some \mbox{\boldmath{$\betaLASSO_j$}} completely to zero, resulting in a sparser estimate. In the case of $\bX$ being an~orthogonal matrix, i.e. $\frac{1}{n}\bX'\bX=\bI_p$, the exact formula for $\hat\bbeta^{LASSO}$ introduced by Tibshirani~\cite{tibshirani1996regression} is based on $\hat\bbeta^{OLS}$:
	$$
	\hat\beta^{LASSO}_i = \sign(\hat\beta^{OLS}_i)\max\{|\hat\beta^{OLS}_i|-\lambda, 0\}.
	$$
	Another approach to reduce the dimensionality is the Sorted $\ell_1$ Penalized Estimator\\ (SLOPE~\cite{MBtheory,MBslope,zeng2014}), which not only generalizes the LASSO method, but also allows to clusterize the similar coefficients of $\bbeta$. In SLOPE, $\ell_1$-norm is replaced by its sorted version $J_{\mathbf \Lambda}$, which depends on the tuning vector $\bLambda=(\lambda_1,\ldots,\lambda_p)\in\R^p$, where $\lambda_1 \ge \ldots \ge \lambda_p \ge 0$:
	$$
	J_{\mathbf \Lambda}(\bbeta):=\sum\limits_{i=1}^{p} \lambda_i |\bbeta|_{(i)},
	$$
	where $\{|\bbeta|_{(i)}\}_{i=1}^{p}$ is a decreasing permutation of absolute values of $\bbeta_1,\ldots,\bbeta_p$:
	\begin{equation*}
		\hat\bbeta^{SLOPE} := \operatorname*{arg\,min}_{{\mathbf b}\in\R^p} \left[\frac12\|\bY-\bX \bb\|^{2}_{2} + J_{\mathbf \Lambda}(\bb)\right].
	\end{equation*}
	The case of $\bLambda$ being an arithmetic sequence was studied by~Bondell and~{Reich}~\cite{bondell2008} and called the Octagonal Shrinkage and Clustering Algorithm for~Regression (OSCAR). The special case of SLOPE with $\lambda_1 = \lambda_2 = \ldots = \lambda_p > 0$ is~LASSO. For $\bLambda=(0,\ldots,0)$ we obtain the OLS estimator.\\
	Clustering the predictors allows for additional dimension reduction by identifying variables with the same absolute values of regression coefficients.
	One may recently observe the rise of interest in methods, which cluster highly correlated 
	predictors~\cite{bondell2009,gertheiss,majkanska,pokaretal,oelker,scope}. SLOPE is ideal for~this task, since it is capable to identify the low-dimensional structure, which is~called the~SLOPE pattern, defined by Schneider and Tardivel with the subdifferential of the SLOPE norm $J_{\mathbf \Lambda}$~\cite{geom1}. For~the~convention of this article we let $\sign(0)=0$. 
	As $k$ we will denote the number of~clusters of $\pattern(\bb)=(m_1,\ldots,m_p)'$ i.e., the number of nonzero components of $|\bb|$. 
	
	\begin{definition}[SLOPE pattern~\cite{modrec}]
		The SLOPE pattern is a function\\ $\pattern:\R^p \to \Z^p$ such that
		\begin{equation*}
			patt(\bb)_{i} = \sign(b_i) rank(|b_i|),
		\end{equation*}
		where $rank(|b_i|)\in\{1,2,\ldots,k\}$ is a rank of $|b_i|$ in a vector of distinct nonzero values among $\{|b_1|,\ldots,|b_p|\}$. We adopt the convention that $rank(0)=0$.
	\end{definition}
	As $\mathcal{M}_p$ we denote the set of all possible SLOPE patterns of $\bb\in\R^p$.
	\begin{fact}[Basic properties of SLOPE pattern~\cite{geom1}]
		\leavevmode
		\makeatletter
		\@nobreaktrue
		\makeatother
		\begin{enumerate}[label=(\alph*)]
			\item for every $1\le l\le\|\pattern(\bb)\|_\infty$ there exists $j$ such that $|patt(\bb)_j|=l$,
			\item $\mbox{\rm sign}(\pattern(\bb)) = \sign(\bb)$ (sign preservation),
			\item $|b_i|=|b_j|\Then|patt(\bb)_i|=|patt(\bb)_j|$ (cluster preservation),
			\item $|b_i|>|b_j|\Then|patt(\bb)_i|>|patt(\bb)_j|$ (hierarchy preservation).
		\end{enumerate}
	\end{fact}
	\noindent
	\begin{example}
		$\pattern(4,0,-1.5,1.5,-4)=(2,0,-1,1,-2)$.
	\end{example}
	\begin{remark}[Subdifferential description of the SLOPE pattern~\cite{geom1}]
		{\;}\\
		Let $\bLambda=(\lambda_1,\ldots,\lambda_p)$ satisfy $\lambda_1 > \ldots > \lambda_p > 0$. Then
		$$
		\pattern(\bb_1)=\pattern(\bb_2) \iff \partial_{J_{\mathbf \Lambda}}(\bb_1) = \partial_{J_{\mathbf \Lambda}}(\bb_2),
		$$
		where $\partial_{f}(\bb)$ is a subdifferential of the function $f:\R^p\to\R$ at $\bb$, i.e.:
		$$
		\partial_{f}(\bb) = \{v\in\R^p: f(\bz)\ge f(\bb)+v'(\bz-\bb)\quad \forall \bz\in\R^p\}.
		$$
	\end{remark}
	The subdifferential approach may be applied to a wider class of penalizers being polyhedral gauges, cf.~\cite{geom2}.

	\begin{definition}[{Pattern recovery by SLOPE}]
		We say that the SLOPE estimator $\widehat{\bbeta}^{SLOPE}$ recovers the pattern of $\bbeta$ when 
		$$
		\pattern\left(\widehat{\bbeta}^{SLOPE}\right)=\pattern(\bbeta).
		$$
	\end{definition}
	The clustering properties of SLOPE have been studied before, cf.~\cite{bondell2008,nowak}, but the researchers consider strongly correlated predictors, which are being used in financial mathematics to~group the assets with respect to their partial correlation with the hedge fund return times series~\cite{kremer}. 
	In our article we decided to suppose the orthogonal design
	\begin{equation}\label{eq:design}
		\bX'\bX=n\mathbb{I}_p.
	\end{equation}
	This is a classical and natural assumption in the case where n and p are fixed, cf.~\cite{tibshirani1996regression}. Moreover, in the asymptotic case, where $n \to \infty$ and $p$ is fixed, it is usually supposed that $\bX'\bX/n \to \bC>0$, cf.~\cite{zhao,zou}. In~\eqref{eq:design} the design matrix $\bX$ is orthogonal.
	Then, the Euclidean norm of each $n$-dimensional column of $\bX$ equals $n$. If  it was $1$, the terms of $\bX$ would vanish to zero for large $n$, which is not natural.
	Such class of matrices is being widely used in signal analysis,~\cite{signalON,cds}. For general $\bX$ the~problem is~considered in our parallel article~\cite{modrec}.\\
	To study the properties of SLOPE we often use the closed unit ball $C_{\mathbf \Lambda}$ in the dual norm of $J_{\mathbf \Lambda}$, which was studied e.g. by Zeng and Figueiredo~\cite{zeng2014}. This dual ball is described explicitely as a signed permutahedron, see e.g.~\cite{negrinho,geom1}.
	\begin{equation}\label{dualball}
		C_{\mathbf \Lambda}=\{\bpi=(\pi_1,\pi_2,\ldots,\pi_p)\in\R^p: \sum_{j \leq i} |\pi|_{(j)} \leq 
		\sum_{j \leq i} \lambda_j: i=1,2,\ldots,p \}.
	\end{equation}
	
	In this article we prove novel results on the strong consistency of SLOPE both in estimation and in pattern recovery. We also introduce a new, based on minimaxity, approach to relations between $\hat\bbeta^{SLOPE}$ and $\hat\bbeta^{OLS}$.

	\subsection{Outline of the paper}
	In Section~\ref{sec:minimax} we derive the connections between $\hat\bbeta^{SLOPE}$ and $\hat\bbeta^{OLS}$ in the orthogonal design. We use the minimax theorem of Sion, cf.~\cite{aubin}.
	In Section \ref{subsec:ONslope} we focus on the properties of~$\hat\bbeta^{SLOPE}$. 
	We use the geometric interpretation of SLOPE to explain its ability to identify the SLOPE pattern and provide new theoretical results on the support recovery and clustering properties using a representation of SLOPE as a function of~the~ordinary least squares (OLS) estimator. Similar approach for LASSO was used by Ewald and Schneider, cf.~\cite{ewald}.

	
	To analyze the asymptotic properties of the SLOPE estimator, e.g. its consistency, we~have to assume that the sample size $n$ tends to infinity. Therefore, in~Section \ref{sec:asymp} we define a~sequence of linear regression models 
	$$\bY^{(n)}=\bX^{(n)}\bbeta + \beps_n^{(n)}.$$ 
	In this sequence, the response vector $\bY^{(n)} \in \R^{n}$, the design matrix $\bX^{(n)}\in\R^{n\times p}$ and~the~error term $\beps^{(n)}=(\eps_1^{(n)},\eps_2^{(n)}, \ldots,\eps_n^{(n)})' \in \R^n$ vary with $n$.

	The error term $\beps^{(n)}=(\eps_1^{(n)},\eps_2^{(n)}, \ldots,\eps^{(n)})' \in \R^n$ has the normal distribution $N(0, \sigma^2\bI_n)$. We make no assumptions about the relations between $\beps^{(n)}$ and $\beps^{(m)}$ for $n\neq m$.  
	In this paper we consider the specific, but statistically important model in which $n\geq p$ and columns of~$\bX$ are orthogonal. The orthogonality assumption allows us to derive, by simple techniques, relatively precise results on the SLOPE estimator (e.g. Theorem 3.1 theorem), which seems unavailable when columns of~$\bX$ are not orthogonal.\\ 
	Substantially more difficult techniques based on subdifferential calculus are developed in~\cite{modrec}. These techniques are used in~\cite{modrec} to establish the properties of the SLOPE estimator in the general case, where the columns of $\bX$ are not orthogonal and $p$ may be much larger than $n$. In the asymptotic theorem proved in~\cite{modrec} under different assumptions on $\bX_n'\bX_n$ stronger restrictions on tuning $\lambda_n$ are considered.
	We provide the~conditions under which the SLOPE estimator is~strongly consistent. Additionally, in~case when for each $n$ the design matrix is~orthogonal, we~provide the conditions on~the~sequence of~tuning parameters such that SLOPE is~strongly consistent in the pattern recovery. 
	In Section~\ref{sec:appsim} we show the applications of the SLOPE clustering in terms of high frequency signal denoising and illustrate them with simulations.
	The Appendix covers the proofs of technical results.	
	
	\section{Approach by minimax theorem}\label{sec:minimax}
	
	\subsection{Technical results}
	Let $r_{SLOPE}$ denote the minimum value of  the   SLOPE criterion, attained by $\hat\bbeta^{SLOPE}$, i.e. 
	$$
	r_{SLOPE}:=\min_{{\mathbf b} \in \mathbb{R}^p}  \left [ \,\frac{1}{2} \|\bY-\bX\bb \|_2^2+ J_{\mathbf \Lambda}(\bb)\, \right ]= \frac{1}{2} \|\bY-\bX\widehat{\bbeta}^{SLOPE} \|_2^2+ J_{\mathbf\Lambda}(\widehat{\bbeta}^{SLOPE}).
	$$
	Since 
	$$
	\|\widehat{\bbeta}^{SLOPE} \|_{2} \leq \sqrt{p} \|\widehat{\bbeta}^{SLOPE} \|_{\infty} \quad {\rm and} \quad \lambda_1 \|\widehat{\bbeta}^{SLOPE} \|_{\infty} \leq J_{\mathbf \Lambda}(\widehat{\bbeta}^{SLOPE}) \leq r_{SLOPE},
	$$
	it follows that
	\begin{equation*}
		\lambda_1 \left\|\widehat{\bbeta}^{SLOPE} \right\|_{2} \leq \sqrt{p}\ r_{SLOPE} \leq \sqrt{p} \left[ \frac{1}{2} \|\bY-\bX\bnull \|_2^2+ J_{\mathbf \Lambda}(\bnull)   \right] = \frac{\sqrt{p}}{2} \left\|\bY\right\|_2^2.
	\end{equation*}
	We immediately get the following result.
	\begin{corollary} 
		$\left\|\widehat{\bbeta}^{SLOPE} \right\|_{2}^2 \leq M_0$, where $M_0 = \left(\frac{p \,\|\bY\|_2^4}{4 \lambda_1^2} \right)$.
	\end{corollary}
	From this corollary it is seen that we
	can clearly limit our search to vectors $\bbeta$ from the~compact set   ${\cal M} \subset \mathbb{R}^p$ defined by ${\cal M}:=\left \{ \bb \in \mathbb{R}^p:  \|\bb\|_{2}^2 \leq M_0  \right \}$.
	Therefore, we can equivalently define a SLOPE solution by
	\begin{equation}\label{slope1}
		\widehat{\bbeta}^{SLOPE} = \operatorname*{arg\,min}_{{\mathbf b} \in {\cal M}}  \left [ \,\frac{1}{2} \|\bY-\bX\bb \|_2^2+J_{\mathbf \Lambda}(\bb)\, \right ]. 
	\end{equation}
	
	\begin{proposition}\label{lemma1} 
		Let $C_{\mathbf \Lambda}$ be the unit ball in the~dual SLOPE norm. Then, for~each $\bb \in \mathbb{R}^p$, 
		\begin{equation}\label{eq:unitball}
			\displaystyle
			J_{\mathbf \Lambda}(\bb) =\max_{{\mathbf \pi} \in C_{\mathbf \Lambda}} \bpi' \bb.
		\end{equation}
	\end{proposition}
	The proof is a simple application of the definition of the dual norm and the reflexivity of~$(\R^p,J_{\mathbf{\Lambda}}) = (\R^p,J_{\mathbf{\Lambda}}^{*})^{*}$. Thus
	$$
	J_{\mathbf{\Lambda}}(\bb) = \|\bb\|_{(\R^p,J_{\mathbf{\Lambda}})} = \sup\limits_{\mathbf{x}:J_{\mathbf{\Lambda}}^{*}(\mathbf{x})\leq 1}\mathbf{x}'\bb.
	$$
	\begin{remark}
		\begin{enumerate}
			\item[(a)] A different, longer proof is given in~\cite[Proposition 1.1]{MBtheory}
			\item[(b)] The formula~\eqref{eq:unitball} holds in much greater generality of Lov\'asz extensions in place of the $J_{\mathbf{\Lambda}}$ norm, see~\cite{minami}. We thank an anonymous referee for noticing this fact.
		\end{enumerate}
	\end{remark}
	\subsection{Saddle point} \label{Section Saddle point}
	Let  the function $r:{\cal M} \times C_{\mathbf \Lambda} \rightarrow \mathbb{R}$ be defined by
	$$
	r(\bb,\bpi):= \frac{1}{2} \|\bY-\bX\bb \|_2^2+ \bpi' \bb.
	$$
	As an immediate consequence of (\eqref{slope1}) and  {Proposition \ref{lemma1}}   we obtain
	\begin{align*} 
			r_{SLOPE}&=&\min_{{\mathbf b} \in \mathbb{R}^p}  \left [ \,\frac{1}{2} \|\bY-\bX\bb \|_2^2+ J_{\mathbf \Lambda}(\bb)\, \right ] = \min_{{\mathbf b} \in {\cal M}}  \left [ \,\frac{1}{2} \|\bY-\bX\bb \|_2^2+ J_{\mathbf \Lambda}(\bb)\, \right ]
		\\  
		& = &\min_{{\mathbf b} \in {\cal M}} \max_{{\mathbf \pi} \in C_{\mathbf \Lambda}}  \left [ \,\frac{1}{2} \|\bY-\bX\bb \|_2^2+ \bpi' \bb\, \right ] = \min_{{\mathbf b} \in {\cal M}} \max_{{\mathbf \pi} \in C_{\mathbf \Lambda}} \, r(\bb,\bpi).  
	\end{align*}
	It turns out that the order of the maximization over $\bpi \in C_{\mathbf \Lambda}$ and the minimization over $\bb  \in {\cal M}$ can be switched without affecting the result. To see this, note that both $C_{\mathbf \Lambda}$ and ${\cal M}$  are convex and compact. Moreover,
	for each fixed $\bpi \in C_{\mathbf \Lambda}$, $r(\bb,\bpi)$ is~a~convex continuous function with respect to
	$\bb \in {\cal M}$ and, for each fixed $\bb \in{\cal M}$,
	$r(\bb,\bpi)$ is concave  with respect to~$\bpi \in C_{\mathbf \Lambda}$ (in fact, it is linear).
	Therefore, all assumptions of the Sion's minimax  theorem  are fulfilled (see~\cite[p. 218]{aubin}) and thus there exists a saddle point  $(\bbeta^*, \bpi^*) \in {\cal M} \times  C_{\mathbf \Lambda}$ such that 
	\begin{eqnarray*}
		\max_{{\mathbf \pi} \in C_{\mathbf \Lambda}} \min_{{\mathbf b} \in {\cal M}}\; r(\bb,\bpi) & = & \min_{{\mathbf b} \in {\cal M}}\; r(\bb,\bpi^*) = 
		r(\bbeta^*,\bpi^*) \\
		&=&  \max_{{\mathbf \pi} \in C_{\mathbf \Lambda}} \; r(\bbeta^*,\bpi)  = 
		\min_{{\mathbf b} \in {\cal M}} \max_{{\mathbf \pi} \in C_{\mathbf \Lambda}} \; r(\bb,\bpi) =r_{SLOPE}. \nonumber
	\end{eqnarray*}
	In the next section we shall see that the first coordinate of  any saddle point $(\bbeta^*, \bpi^*)$ is~the~SLOPE estimator.
	
	\subsection{SLOPE solution for the orthogonal design}
	\label{Section invertible}
	Since for each fixed $\bpi \in C_{\mathbf \Lambda}$, the function $r(\bb,\bpi)$ is convex with respect to $\bb \in {\cal M}$, 
	any~point $\bb_{\bpi} \in {\cal M}$, at which the gradient $\displaystyle \frac{\partial r(\bb,\bpi) }{\partial \bb}$ is zero, is a global minimum. 
	If we rewrite $r(\bb,\bpi)$ as
	\begin{equation*}
		r(\bb,\bpi) = \frac{1}{2}\bY' \bY-\bY'\bX \bb+ \frac{1}{2}\bb'\bX'\bX \bb+ \bpi' \bb 
	\end{equation*} 
	and differentiate with respect to $\bb$, we obtain 
	$$
	\frac{\partial r(\bb,\bpi) }{\partial{\bb}} = -\bX'(\bY - \bX\bb)+\bpi. 
	$$
	Equating this gradient to $\bnull$  gives the following equation for the optimum point $\bb_{\bpi}$:
	\begin{equation}\label{gradient}
		\bX'\bX\bb_{\bpi}= \bX'\bY - \bpi.
	\end{equation}
	Substituting this into the equation for $r(\bb_{\bpi},\bpi)$, we find that
	\begin{eqnarray*}
		r(\bb_{\bpi},\bpi)&=&\frac{1}{2}\bY' \bY-\bb_{\bpi}'\bX' \bY+ \frac{1}{2}\bb_{\bpi}'\bX'\bX \bb_{\bpi}+\bpi' \bb_{\bpi} \\
		&=&  \frac{1}{2}\bY' \bY-\bb_{\bpi}'\bX' \bY+ \bb_{\bpi}'\bX'\bX \bb_{\bpi}+\bb_{\bpi}' \bpi -
		\frac{1}{2}\bb_{\bpi}'\bX' \bX \bb_{\bpi}  \\ 
		& = & \frac{1}{2}\bY' \bY-
		\frac{1}{2}\bb_{\bpi}'\bX'\bX \bb_{\bpi} = \frac{1}{2}\bY' \bY-
		\frac{1}{2}\bb_{\bpi}'\bX'\bX(\bX'\bX)^{-1}\bX'\bX \bb_{\bpi} \\
		&=& \frac{1}{2}\bY' \bY-\frac{1}{2}(\bX'\bY - \bpi)'(\bX'\bX)^{-1}(\bX'\bY - \bpi).
	\end{eqnarray*}
	Let $p_j=|\{i: |m_i|=k+1-j\}|$ be the number of elements of the $j^{th}$ cluster of~$\bbeta$, $P_j=\sum\limits_{i\le j}p_i$ and $P_{k+1}=p$. 
	\begin{lemma}  \label{lemma 2 bis}
		Let $\bpi^*=(\pi_1^*,\ldots,\pi_p^*)' \in C_{\mathbf \Lambda}$ be any solution of
		\begin{equation*}
			\bpi^*=\operatorname*{arg\,min}_{\mathbf \pi \in C_{\mathbf \Lambda}} \left[ (\bX'\bY - \bpi)'(\bX'\bX)^{-1}(\bX'\bY-\bpi) \,\right ]
		\end{equation*}
		and let $\bbeta^*=(\beta_1^*,\ldots,\beta_p^*)'$ be the corresponding point from ${\cal M}$ given by
		\begin{equation*}
			\bbeta^*= (\bX'\bX)^{-1}(\bX'\bY - \bpi^*). 
		\end{equation*}
		Then, $(\bpi-\bpi^*)' \bbeta^*   \leq  0$,  for all  $\bpi \in C_{\mathbf \Lambda}$ and hence 
		\begin{enumerate}
			\item[(a)] $\mbox{\rm sign}\left(\beta_i^*\right) \cdot \mbox{\rm sign}\lt(\pi_i^*\rt) \ge 0$, $i=1,2,\ldots,p$,
			\item[(b)] $\lt(|\pi_1^*|,\ldots,|\pi_p^*|\rt)$ and~$\lt(|\beta_1^*|,\ldots,|\beta_p^*|\rt)$ are similarly sorted, i.e.\\ if $|(patt(\beta))_i|=k+1-j$, then $\lt|\pi^*\rt|_{i}\in\lt\{|\pi^*|_{(P_{j-1}+1)},\ldots,|\pi^*|_{(P_{j})}\rt\}$,
			\item[(c)] for any permutation $\tau$ satisfying $|\beta^*_{\tau(1)}|\geq\ldots\geq|\beta^*_{\tau(p)}|$, if there is a $k \in \{2,\ldots,p\}$ such that $\sum_{i=1}^{k-1} \lt|\pi^*_{\tau(i)}\rt|< \sum_{i=1}^{k-1} \lambda_i \quad \mbox{\rm and} \quad \lt|\pi^*_{\tau(k)}\rt|>0$,\\ \noindent
			then $\lt|\beta^*_{\tau(k-1)}\rt|=\lt|\beta^*_{\tau(k)}\rt|$.
		\end{enumerate}
	\end{lemma}
	The proof is given in the Appendix.
	An immediate consequence of the Lemma is the following result.
	\begin{lemma} \label{lemma 3} The point $(\bbeta^*,\bpi^* )$  defined as in Lemma~\ref{lemma 2 bis} is the saddle point of the function $r(\bb,\bpi)$.
	\end{lemma}
	The proof is given in the Appendix.
	We use the last lemma to prove the main result of this section.
	\begin{theorem} \label{th:charactMW} 
		Let the point  $\bbeta^*$ be defined as in Lemma~\ref{lemma 2 bis}. Then  $\bbeta^*$ is the SLOPE estimator of $\bbeta$.
	\end{theorem}
	
	\begin{proof}
		
		Using the fact that  $\displaystyle \max_{{\mathbf \pi} \in C_{\mathbf \Lambda}} \; r(\bbeta^*,\bpi) = \min_{{\mathbf b} \in {\cal M}} \max_{{\mathbf \pi} \in C_{\mathbf \Lambda}} \; r(\bb,\bpi)$ (see previous lemma) we have  
		\begin{eqnarray*} 
			\lefteqn{ \frac{1}{2} \left\|\bY-\bX\bbeta^* \right\|_2^2+ J_{\mathbf \Lambda}(\bbeta^*) = 
				\max_{{\mathbf \pi} \in C_{\mathbf \Lambda}} \left [ \,\frac{1}{2} \left\|\bY-\bX\bbeta^* \right\|_2^2+ \bpi' \bbeta^*\, \right ] }\hspace{5mm} \\ 
			&=&  \max_{{\mathbf \pi} \in C_{\mathbf \Lambda}} \; r(\bbeta^*,\bpi) 
			=  \min_{{\mathbf b} \in {\cal M}} \max_{{\mathbf \pi} \in C_{\mathbf \Lambda}} \; r(\bb,\bpi) 
			=\min_{{\mathbf b} \in \mathbb{R}^p}  \left [ \,\frac{1}{2} \left\|\bY-\bX\bb \right\|_2^2+ J_{\mathbf \Lambda}(\bb)\, \right ].\qedhere
		\end{eqnarray*}
	\end{proof}
	\begin{corollary}\label{cor:SLOPEvsOLS}
		In the linear model satisfying $\frac{1}{n}\bX'\bX=\bI_p$ we have
		$$
		\hat\bbeta^{OLS}-\hat\bbeta^{SLOPE}=\frac{1}{n}\bpi^* = \frac{1}{n} \operatorname*{arg\,min}\limits_{\mathbf \pi\in C_{\mathbf \Lambda}} \left\|\hat\bbeta^{OLS}-\bpi\right\|^2_2 = \operatorname*{arg\,min}\limits_{\mathbf \pi\in C_{\mathbf \Lambda/n}} \left\|\hat\bbeta^{OLS}-\bpi\right\|^2_2,
		$$
		is the proximal projection of $\hat\bbeta^{OLS}$ onto $C_{\mathbf \Lambda/n}$.
	\end{corollary}
	Projections onto $C_{\mathbf{\Lambda}}$ are widely used in~\cite{minami} in the study of the notion of degrees of freedom. However, the Corollary~\ref{cor:SLOPEvsOLS} is not stated there explicitely.
	\begin{remark}\label{m0m1}
		For each $\bpi \in C_{\mathbf \Lambda}$, the point $\bb_{\mathbf \pi}$ defined in (\ref{gradient}) should be in 
		$$
		\left \{ \bb \in \mathbb{R}^p:  \|\bb\|_{2}^2 \leq M  \right \},
		$$
		where $M$ is chosen so that $M>\max\{M_0,M_1 \}$ with
		$$
		M_1:=\max_{{\mathbf \pi} \in C_{\mathbf \Lambda}} \| (\bX'\bX)^{-1} (\bX'\bY - \bpi)  \|_{2}^2 \leq M.
		$$
	\end{remark}

	\section{Properties of SLOPE in the orthogonal design} 
	\label{subsec:ONslope}
	\subsection{SLOPE vs. OLS}
	By the Theorem~\ref{th:charactMW} and Corollary~\ref{cor:SLOPEvsOLS}, when\\ $\frac1n \bX'\bX=\bI_p$, the orthogonal projection of~the~ordinary least squares estimator\\ ${\hat\bbeta^{OLS}=\frac1n \bX'\bY}$ onto the unit ball $C_{{\mathbf \Lambda}/n}$ is~equal~to $\hat\bbeta^{OLS}-\hat\bbeta^{SLOPE}$.
	For $\bLambda=(200,100)'$ and $n=50$ this property is illustrated on Figure~\ref{fig:OLSonSLOPE}. 
	The figure represents $\hat \bbeta^{SLOPE}$ (black arrows) depending on the localization of $\hat\bbeta^{OLS}$ in the orthogonal design. For $\hat\bbeta^{OLS}$ being the blue point located on the area labelled by $(1,0)$ the first component of~$\hat \bbeta^{SLOPE}$ is positive and the second is null. 
	For $\hat\bbeta^{OLS}$ being the yellow point located on the area labelled by $(-1,1)$ both components of $\hat \bbeta^{SLOPE}$ have equal absolute value (clusterization), but their signs are opposite. For $\hat\bbeta^{OLS}$ being the~red point located on the area labelled by $(1,2)$  both components of $\hat\bbeta^{SLOPE}$ are positive and the first component is~smaller than the second one. The blue polytope is the dual SLOPE unit ball $C_{\mathbf \Lambda}$ 
	and labels 
	$$
	\mathcal{M}_2 = \{(0,0),(\pm 1,0), (0,\pm 1),(\pm 1,\pm 1),  (\pm 2,\pm 1), (\pm 1,\pm 2)\}
	$$
	associated to the areas of this figure correspond to all SLOPE patterns for\\ $n = 50$ and $p = 2$.
	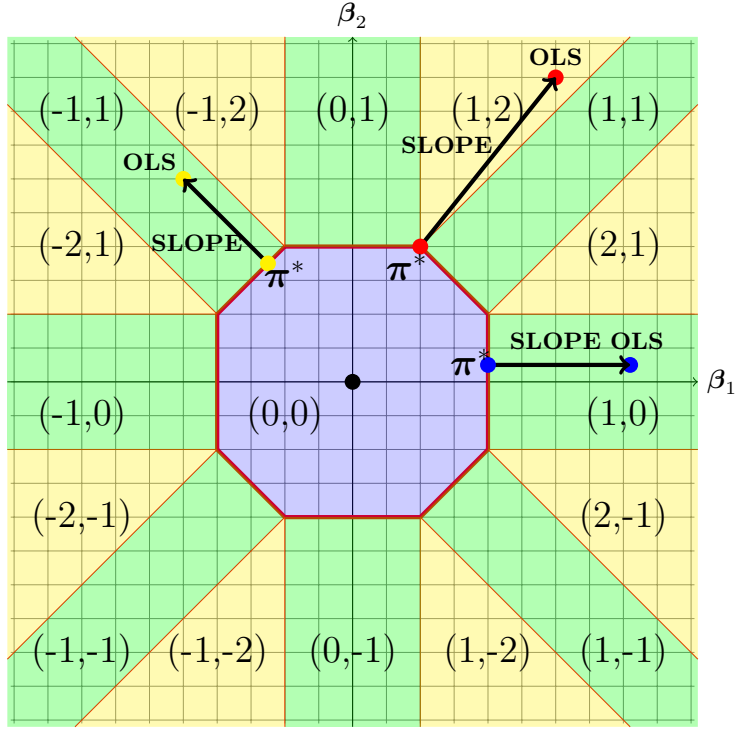
\begin{figure}[ht]
		\centering
		\begin{tikzpicture}[scale=0.45]
			\draw[very thin,color=gray] (-10.1,-10.1) grid (10.1,10.1);
			\draw[->] (-10.2,0) -- (10.2,0) node[right] {$\bbeta_1$};
			\draw[->] (0,-10.2) -- (0,10.2) node[above] {$\bbeta_2$};
			\begin{scope}[ultra thick]
				\draw[color=red] (4,2) -- (2,4) -- (-2,4) -- (-4,2) -- (-4,-2) -- (-2,-4) -- (2,-4) -- (4,-2) -- (4,2);
				\fill[blue,fill opacity=0.2] (4,2) -- (2,4) -- (-2,4) -- (-4,2) -- (-4,-2) -- (-2,-4) -- (2,-4) -- (4,-2) -- (4,2);
			\end{scope}
			\draw[color=red] (4,2) -- (10.2,2); 
			\draw[color=red] (4,-2) -- (10.2,-2);
			\draw[color=red] (4,2) -- (10.2,8.2); 
			\draw[color=red] (2,4) -- (8.2,10.2); 
			\draw[color=red] (2,4) -- (2,10.2); 
			\draw[color=red] (-2,4) -- (-2,10.2); 
			\draw[color=red] (-2,4) -- (-8.2,10.2); 
			\draw[color=red] (-4,2) -- (-10.2,8.2); 
			\draw[color=red] (-4,2) -- (-10.2,2); 
			\draw[color=red] (-4,-2) -- (-10.2,-2); 
			\draw[color=red] (-4,-2) -- (-10.2,-8.2); 
			\draw[color=red] (-2,-4) -- (-8.2,-10.2); 
			\draw[color=red] (-2,-4) -- (-2,-10.2); 
			\draw[color=red] (2,-4) -- (2,-10.2); 
			\draw[color=red] (4,-2) -- (10.2,-8.2); 
			\draw[color=red] (2,-4) -- (8.2,-10.2); 
			\fill[yellow,fill opacity=0.3] (4,2)--(10.2,2)--(10.2,8.2);
			\fill[yellow,fill opacity=0.3] (2,4)--(2,10.2)--(8.2,10.2);
			\fill[yellow,fill opacity=0.3] (-2,4)--(-2,10.2)--(-8.2,10.2);
			\fill[yellow,fill opacity=0.3] (-4,2)--(-10.2,2)--(-10.2,8.2);
			\fill[yellow,fill opacity=0.3] (-4,-2)--(-10.2,-2)--(-10.2,-8.2);
			\fill[yellow,fill opacity=0.3] (-2,-4)--(-2,-10.2)--(-8.2,-10.2);
			\fill[yellow,fill opacity=0.3] (2,-4)--(2,-10.2)--(8.2,-10.2);
			\fill[yellow,fill opacity=0.3] (4,-2)--(10.2,-2)--(10.2,-8.2);
			\fill[green,fill opacity=0.3] (4,2)--(10.2,2)--(10.2,-2)--(4,-2);
			\fill[green,fill opacity=0.3] (4,2)--(10.2,8.2)--(10.2,10.2)--(8.2,10.2)--(2,4);
			\fill[green,fill opacity=0.3] (2,4)--(2,10.2)--(-2,10.2)--(-2,4);
			\fill[green,fill opacity=0.3] (-2,4)--(-8.2,10.2)--(-10.2,10.2)--(-10.2,8.2)--(-4,2);
			\fill[green,fill opacity=0.3] (-4,2)--(-10.2,2)--(-10.2,-2)--(-4,-2);
			\fill[green,fill opacity=0.3] (-4,-2)--(-10.2,-8.2)--(-10.2,-10.2)--(-8.2,-10.2)--(-2,-4);
			\fill[green,fill opacity=0.3] (2,-4)--(2,-10.2)--(-2,-10.2)--(-2,-4);
			\fill[green,fill opacity=0.3] (4,-2)--(10.2,-8.2)--(10.2,-10.2)--(8.2,-10.2)--(2,-4);
			\draw (8.2,0.5) node[color=blue,circle,fill,inner sep=0pt,minimum size=6pt]{};
			\draw [ultra thick,->](4,0.5) -- (8.2,0.5);
			\draw (4,0.5) node[color=blue,circle,fill,inner sep=0pt,minimum size=6pt]{};
			\draw (3.5,0.6) node[color=black]{\Large{$\bf{\bpi^{*}}$}};
			\draw (6,1.2) node[color=black]{\small{\bf{SLOPE}}};
			\draw (8.4,1.2) node[color=black]{\small{\bf{OLS}}};
			\draw (0,0) node[color=black,circle,fill,inner sep=0pt,minimum size=6pt]{};
			\draw (-5,6) node[color=yellow,circle,fill,inner sep=0pt,minimum size=6pt]{};
			\draw [ultra thick,->](-2.5,3.5) -- (-5,6);
			\draw (-2.5,3.5) node[color=yellow,circle,fill,inner sep=0pt,minimum size=6pt]{};
			\draw (-2,3.2) node[color=black]{\Large{$\bf{\bpi^{*}}$}};
			\draw (-4.6,4.1) node[color=black]{\small{\bf{SLOPE}}};
			\draw (-6,6.5) node[color=black]{\small{\bf{OLS}}};
			\draw (6,9) node[color=red,circle,fill,inner sep=0pt,minimum size=6pt]{};
			\draw [ultra thick,->](2,4) -- (6,9);
			\draw (2,4) node[color=red,circle,fill,inner sep=0pt,minimum size=6pt]{};
			\draw (1.6,3.4) node[color=black]{\Large{$\bf{\bpi^{*}}$}};
			\draw (2.79,7) node[color=black]{\small{\bf{SLOPE}}};
			\draw (6,9.6) node[color=black]{\small{\bf{OLS}}};
			\draw (-8,4) node[color=black]{\Large{(-2,1)}};
			\draw (-4,8) node[color=black]{\Large(-1,2)};
			\draw (4,8) node[color=black]{\Large(1,2)};
			\draw (8,4) node[color=black]{\Large(2,1)};
			\draw (8,-4) node[color=black]{\Large(2,-1)};
			\draw (4,-8) node[color=black]{\Large(1,-2)};
			\draw (-4,-8) node[color=black]{\Large(-1,-2)};
			\draw (-8,-4) node[color=black]{\Large(-2,-1)};
			\draw (-8,-1) node[color=black]{\Large(-1,0)};
			\draw (-8,8) node[color=black]{\Large(-1,1)};
			\draw (0,8) node[color=black]{\Large(0,1)};
			\draw (8,8) node[color=black]{\Large(1,1)};
			\draw (8,-1) node[color=black]{\Large(1,0)};
			\draw (8,-8) node[color=black]{\Large(1,-1)};
			\draw (0,-8) node[color=black]{\Large(0,-1)};
			\draw (-8,-8) node[color=black]{\Large(-1,-1)};
			\draw (-2,-1) node[color=black]{\Large(0,0)};
		\end{tikzpicture}
		\caption{The dual unit ball $C_{{\mathbf \Lambda}/n}$ for $\bLambda=(200,100)'$ and examples of $\hat\bbeta^{SLOPE}$ and~$\hat\bbeta^{OLS}$ in~the~orthogonal design for $n=50$ and $p=2$. The labels of each colored set refer to~the~pattern of $\hat\bbeta^{SLOPE}$ for $\hat\bbeta^{OLS}$ lying in this set. The arrows point from $(\hat\bbeta^{OLS}-\hat\bbeta^{SLOPE})$ to~$\hat\bbeta^{OLS}$.}
		\label{fig:OLSonSLOPE}
	\end{figure}
	In~the~orthogonal design, one may also explicitly compute the SLOPE estimator. Indeed, by~the~Corollary~\ref{cor:SLOPEvsOLS}, $\hat \bbeta^{SLOPE}$ is the image of $\hat \bbeta^{OLS}$ by~the~proximal operator of the SLOPE norm. Therefore, this operator has a closed form formula~\cite{MBslope,simple,prox}. This explicit expression gives an analytical way to learn that SLOPE solution is sparse and~built of clusters.
	

	
	\begin{lemma}~\label{lemat 1}In the linear model satisfying $\frac1n \bX'\bX=\bI_p$ we have
			\begin{equation} \label{equality of slopes}
				\operatorname*{arg\,min}_{{\mathbf b} \in \mathbb{R}^p}  \left [ \,\frac{1}{2} \left\|\bY-\bX\bb \right\|_2^2+ J_{\mathbf \Lambda}(\bb)\, \right ]  =  \operatorname*{arg\,min}_{{\mathbf b} \in \mathbb{R}^p}  \left [ \,\frac{1}{2} \left\|\widehat{\bbeta}^{\ols} - \bb \right\|_2^2+ J_{\mathbf \Lambda}(\bb)\, \right ].
			\end{equation}	
		\end{lemma}
		As~\eqref{equality of slopes} is not proven in~\cite[Equation (1.14)]{MBtheory}, we give the proof in the Appendix.
		The next theorem gives a sufficient condition for the clustering effect of the SLOPE estimator in the orthogonal design.
		
		\begin{theorem}
			\label{FAKT 2}
			Consider a linear model with orthogonal design 
			\\ $\frac1n \bX' \bX=\bI_p$. Let  $\pi$ be~a~permutation of $(1,2,\ldots,p)$ such that $$
			\left|\widehat{\bbeta}_{\pi(1)}^{\ols}\right| \ge \left|\widehat{\bbeta}_{\pi(2)}^{\ols}\right| \ge \ldots \ge \left|\widehat{\bbeta}_{\pi(p)}^{\ols}\right|.$$
			For $i \in \{1,2,\ldots,p-1\}$,\\ if $\left|\widehat{\bbeta}_{\pi(i)}^{\ols}\right| - \left|\widehat{\bbeta}_{\pi(i+1)}^{\ols}\right| \leq \frac{\lambda_i - \lambda_{i+1}}{n}$, then $\left|\widehat{\bbeta}^{\slope}_{\pi(i)}\right|=\left|\widehat{\bbeta}^{\slope}_{\pi(i+1)}\right|$.
		\end{theorem}
		\begin{proof}
			By Lemma \ref{lemat 1}, in the orthogonal design, $\widehat{\bbeta}^{\slope}$ is proximal map of $J_{\mathbf{\Lambda}/n}(\cdot)$ at~$\hat\bbeta^{OLS}$. The~result may~be~inferred from~\cite[Lemma 2.3]{MBslope}.
		\end{proof}
		In the following theorem we  derive necessary and sufficient
		conditions under which  SLOPE  in the orthogonal design 
		recovers the support of the vector $\bbeta=(\bbeta_1,\ldots, \bbeta_p)'$, i.e.
		$$\quad\widehat{\bbeta}^{\slope}_i=0 \iff \bbeta_i=0.$$
		
		\begin{theorem}
			\label{supportON}
			Under orthogonal design $\frac1n \bX'\bX=\bI_p$, let  $\pi$ be a permutation of $(1,2,\ldots,p)$ satisfying $\lt|\widehat{\bbeta}_{\pi(1)}^{\ols}\rt| \ge |\widehat{\bbeta}_{\pi(2)}^{\ols}| \ge \ldots \ge |\widehat{\bbeta}_{\pi(p)}^{\ols}|$. Without loss of generality suppose that $supp(\bbeta)=\{1,2,\ldots,p_0\}$ with $p_0<p$. 
			The necessary and sufficient condition for SLOPE to identify the~set of~relevant covariables is:
			\begin{enumerate}
				\item[]{(a)}  \label{a} $\displaystyle \min_{1 \leq i \leq p_0} \lt|\widehat{\bbeta}_{i}^{\ols}\rt| > \max_{p_0+1 \leq i \leq p} \lt|\widehat{\bbeta}_{i}^{\ols}\rt|$, 
				\item[]{(b)} \label{b} $\sum\limits_{i=k}^{p_0} \lt|\widehat{\bbeta}_{\pi(i)}^{\ols}\rt| > \frac1n \sum\limits_{i=k}^{p_0} \lambda_i, \quad \mbox{\rm for} \quad k=1,2,\ldots,p_0$,
				\item[]{(c)} $\sum\limits_{i=p_0+1}^k \lt|\widehat{\bbeta}_{\pi(i)}^{\ols}\rt| \leq \frac1n \sum\limits_{i=p_0+1}^k \lambda_i, \quad \mbox{\rm for} \quad k=p_0+1,p_0+2,\ldots,p.$  
			\end{enumerate}
		\end{theorem} 
		
		\begin{proof}
			The result may be inferred from the properties of the proximal SLOPE 
			\cite[Lemma 2.3 and Lemma 2.4]{MBtheory} and from Lemma \ref{lemat 1}.
		\end{proof}	
		
		\section{Asymptotic properties of SLOPE}
		\label{sec:asymp}
		In this section we discuss several asymptotic properties of SLOPE estimators in the low-dimensional regression model in which $p$ is  fixed   and the sample size $n$ tends to infinity.
		For~each $n \ge 1$ we consider a linear model 
		\begin{equation}\label{reg2}
			\bY^{(n)}=\bX^{(n)}\bbeta + \beps^{(n)},
		\end{equation}
		where $\bY^{(n)}=(y_1^{(n)},y_2^{(n)},\ldots,y_n^{(n)})' \in \R^{n}$ is a vector of observations, $\bX^{(n)}\in\R^{n\times p}$ is a~deterministic design matrix with $\mbox{rank}(\bX^{(n)})=p$, $\bbeta =(\beta_1,\beta_2,\ldots,\beta_p)' \in \R^p$ is a vector of~unknown regression coefficients and  $\beps^{(n)}=(\eps^{(n)}_1,\eps^{(n)}_2, \ldots,\eps^{(n)}_n)'~\in~\R^n$ is a noise term, which has the~normal distribution $N(0, \sigma^2 \bI_n)$. We make no assumptions about the dependence between $\beps^{(n)}$ and $\beps^{(m)}$ for $n\neq m$. In particular, $\beps^{(n)}$ does not need to~be~a~subsequence of~$\beps^{(m)}$.
		\\
		When defining  the sequence $(\widehat{\bbeta}_n^{SLOPE})$ of SLOPE estimators,  we  assume that  the tuning vector varies with $n$. More precisely, for~each $n \ge 1$ its coefficients $\lambda_{1}^{(n)}~\ge~\lambda_{2}^{(n)}~\ge~\ldots~\ge~\lambda_{p}^{(n)}~\ge~0$ are fixed and $\lambda_{1}^{(n)}>0$.
		By $\widehat{\bbeta}_n^{SLOPE}$ we denote the SLOPE estimator corresponding~to the~tuning vector $\bLambda^{(n)}=(\lambda_{1}^{(n)},\ldots,\lambda_{p}^{(n)})'$:
		\begin{equation} \label{slope for consistency}
			\widehat{\bbeta}_n^{SLOPE} = \operatorname*{arg\,min}_{{\mathbf b} \in \mathbb{R}^p}  \left [ \,\frac{1}{2} \lt\|\bY^{(n)}-\bX^{(n)}\bb \rt\|_2^2+ J_{{\mathbf \Lambda}^{(n)}}(\bb)\, \right ]. 
		\end{equation}
		
		\subsection{Strong consistency of the SLOPE estimator}\label{sec:StrongCons}
		Let us recall the definition of a strongly consistent estimator \mbox{\boldmath{$\betaSLOPE_n$}} of $\bbeta$, i.e. $\forall \bbeta\in \R^p$ we have \mbox{\boldmath{$\betaSLOPE_n$}}$\to \bbeta$ almost surely.
		
		Below we discuss consistency of the sequence $(\widehat{\bbeta}_n^{SLOPE})$ of SLOPE estimators, defined by~\eqref{slope for consistency}. 
		
		\begin{theorem}\label{strong consistency of slope} Consider the linear regression model \eqref{reg2} and assume that 
			$$
			\lim\limits_n n^{-1} \lt(\bX^{(n)}\rt)'\bX^{(n)} = \bC,
			$$ 
			where $\bC$ is a positive definite matrix. Let $\widehat{\bbeta}_n^{SLOPE}$, $n \ge 1$,  be the SLOPE estimator corresponding to the tuning vector $\bLambda^{(n)}=(\lambda_{1}^{(n)},\lambda_{2}^{(n)},\ldots,\lambda_{p}^{(n)})'$.
			\begin{enumerate}[label=(\alph*)]
				\item If $\displaystyle \lim_{n \rightarrow \infty}  \frac{\lambda_{1}^{(n)}}{n} = 0$, then $\widehat{\bbeta}^{SLOPE}_{n}\stackrel{a.s.}{\longrightarrow} \bbeta$.
				\item If $\displaystyle \lim_{n \rightarrow \infty}  \frac{\lambda_{1}^{(n)}}{n} = \lambda_0>0$, then $\widehat{\bbeta}^{SLOPE}_{n}$ is not strongly consistent for $\bbeta$.
			\end{enumerate}
			
		\end{theorem}
		Before proving the above theorems we start with stating a simple technical lemma. It follows quickly from the Borel-Cantelli Lemma and the tail inequality:\\
		If $Z\sim\mathrm{N}(0,1)$, then $\P(Z>t)\leq t^{-1}e^{-t^2/2}/\sqrt{2\pi}$,\quad $t>0$.
		\begin{lemma}\label{lem:gauss}
			Assume that $(Q_n)_{n\in\mathbb{N}}$ is a sequence of Gaussian random variables, defined on~the~same probability space, which converges in distribution to~$\mathrm{N}(0,\sigma^2)$ for some $\sigma\in(0,\infty)$. Then, for any $\delta>0$,
			\[
			\lim_{n\to\infty} \frac{Q_n}{(\log(n))^{1/2+\delta}}=0\quad\mbox{a.s.}
			\]
		\end{lemma}
		Our proof of the strong consistency of SLOPE is based on the strong consistency of~the~OLS estimator. The latter result  is a folklore and we prove it in our~setting. 
		\begin{proposition}
			Consider the linear regression model \eqref{reg2}.\\
			If $\lim\limits_n n^{-1} (\bX^{(n)})'\bX^{(n)} =\bC$, where $\bC$ is positive definite, then $\widehat{\bbeta}^{OLS}_{n}\stackrel{a.s.}{\longrightarrow} \bbeta$.
		\end{proposition}
		\begin{proof}
			We have 
			\begin{align*}
				\widehat{\bbeta}^{OLS}_n -\bbeta= ((\bX^{(n)})'\bX^{(n)})^{-1} (\bX^{(n)})' \bY^{(n)} -\bbeta =((\bX^{(n)})' \bX^{(n)})^{-1} (\bX^{(n)})' \beps^{(n)}.
			\end{align*}
			Then $\sqrt{n}\left(\widehat{\bbeta}^{OLS}_n  - \bbeta\right)$ has the normal distribution $N(0, \sigma^2 (n^{-1}(\bX^{(n)})' \bX^{(n)})^{-1})$ and its components satisfy the assumptions of Lemma \ref{lem:gauss}. Since $\log(n)^{1/2+\delta}=o(\sqrt{n})$, we get the assertion by Lemma \ref{lem:gauss}.
		\end{proof}

		\begin{proof}[Proof of Theorem {\ref{strong consistency of slope}}] (a) It follows from Theorem \ref{lemma 2 bis} that there exists a vector $\bpi^*_n \in C_{({\mathbf \Lambda}^{(n)})}$ such that  
			\begin{equation*}
				\widehat{\bbeta}^{SLOPE}_{n}=((\bX^{(n)})' \bX^{(n)})^{-1}((\bX^{(n)})' \bY^{(n)} -\bpi^*_n).
			\end{equation*}
			
			Since  $\bpi^*_n$ takes values in $ \bC_{{\mathbf \Lambda}^{(n)}}$, it follows that  $\|\bpi^*_n \|_{\infty} \leq \lambda_{1}^{(n)}$. Hence,
			\begin{equation} \label{eq: convergence of pi}
				\frac{\bpi^*_n}{n} \stackrel{a.s.}{\longrightarrow} \bnull,
			\end{equation}
			because $ \displaystyle 
			\left \| \frac{\bpi^*_n}{n}  \right \|_{\infty} \leq \frac{\lambda_{1}^{(n)}}{n}  \rightarrow 0$. 
			The assumption that $\mbox{rank}(\bX^{(n)})=p$ implies that the matrix  $(\bX^{(n)})'\bX^{(n)}$ is invertible  and hence  the least squares estimator of $\bbeta$ is~unique and has the~form $\displaystyle 
			\widehat{\bbeta}^{OLS}_{n}=((\bX^{(n)})' \bX^{(n)})^{-1}(\bX^{(n)})'\bY^{(n)}$.
			Combining  with (\eqref{eq: convergence of pi}) the fact that $\widehat{\bbeta}^{OLS}_{n}\stackrel{a.s.}{\longrightarrow} \bbeta$, we~conclude that
			\begin{eqnarray*}
				\widehat{\bbeta}^{SLOPE}_{n} &=& ((\bX^{(n)})' \bX^{(n)})^{-1}((\bX^{(n)})' \bY^{(n)} -\bpi^*_n) = \widehat{\bbeta}^{OLS}_{n} - ((\bX^{(n)})' \bX^{(n)})^{-1}\bpi^*_n \\
				&=&  \widehat{\bbeta}^{OLS}_{n} - \left(\frac{(\bX^{(n)})' \bX^{(n)}}{n} \right)^{-1} \frac{\bpi^*_n}{n} \stackrel{a.s.}{\longrightarrow} \bbeta-\bC^{-1}\bnull= \bbeta.
			\end{eqnarray*}
			\\
			(b) Since $\widehat{\bbeta}^{SLOPE}_{n}$ minimizes  over $\bb \in \mathbb{R}^p$ the function  
			$$
			l(\bb):=\frac{1}{2} \|\bY^{(n)}-\bX^{(n)}\bb \|_2^2+ J_{\mathbf \Lambda^{(n)}}(\bb)
			$$
			and since $\lambda_{1}^{(n)} \|\bb \|_{\infty} \leq J_{\mathbf \Lambda^{(n)}}(\bb)$, it follows that
			\begin{eqnarray*} 
				0 & \le & l(0)-l(\widehat{\bbeta}^{SLOPE}_{n})
				= (\widehat{\bbeta}^{SLOPE}_{n})'(\bX^{(n)})'\bY^{(n)}\\
				& - & \frac{1}{2} (\widehat{\bbeta}^{SLOPE}_{n})'(\bX^{(n)})'\bX^{(n)} \widehat{\bbeta}^{SLOPE}_{n}-J_{\mathbf \Lambda^{(n)}}(\widehat{\bbeta}^{SLOPE}_{n}) \\
				& \leq & (\widehat{\bbeta}^{SLOPE}_{n})'(\bX^{(n)})'\bY^{(n)}-\frac{1}{2} (\widehat{\bbeta}^{SLOPE}_{n})'(\bX^{(n)})'\bX^{(n)} \widehat{\bbeta}^{SLOPE}_{n}\\
				& - & \lambda_{1}^{(n)} \|\widehat{\bbeta}^{SLOPE}_{n} \|_{\infty} = (\widehat{\bbeta}^{SLOPE}_{n})'(\bX^{(n)})'\bX^{(n)} \widehat{\bbeta}^{OLS}_{n}\\
				& - & \frac{1}{2} (\widehat{\bbeta}^{SLOPE}_{n})'(\bX^{(n)})'\bX^{(n)} \widehat{\bbeta}^{SLOPE}_{n}-\lambda_{1}^{(n)} \|\widehat{\bbeta}^{SLOPE}_{n} \|_{\infty}.
			\end{eqnarray*}
			The last equality follows from the fact that $(\widehat{\bbeta}^{SLOPE}_{n})'(\bX^{(n)})'(\bY^{(n)} - \bX^{(n)})\widehat{\bbeta}^{OLS}_{n} = 0$.\\
			Suppose to the contrary that $\widehat{\bbeta}^{SLOPE}_{n}\stackrel{a.s.}{\longrightarrow} \bbeta$. Then, using the facts that\\ $\widehat{\bbeta}^{OLS}_{n} \stackrel{a.s.}{\longrightarrow} \bbeta$ and~that $\lim\limits_n n^{-1} (\bX^{(n)})'\bX^{(n)} =\bC$, we have
			$$
			0 \leq \frac{l(0)-l(\widehat{\bbeta}^{SLOPE}_{n})}{n} \stackrel{a.s.}{\longrightarrow} \bbeta'\bC\bbeta-\frac{1}{2} \bbeta'\bC\bbeta-\lambda_0 \|\bbeta \|_{\infty} = \frac{1}{2} \bbeta'\bC\bbeta-\lambda_0 \|\bbeta \|_{\infty}.
			$$
			For $\lambda_0>0$ this provides a contradiction since   the inequality $\lambda_0 \|\bbeta \|_{\infty} \leq \frac{1}{2} \bbeta'\bC\bbeta$ does~not~hold when the value of $\bbeta$ is sufficiently close to $0$.
		\end{proof}
		
		\begin{remark}
			The proof of Theorem~\ref{strong consistency of slope} (b) does not exclude that 
			$\bbeta^{SLOPE}_n~\to~\bbeta$ for $\|\bbeta\|$ large enough.
			
			However, the definition of strong consistency requires the convergence for any value of the parameter $\bbeta$. We prove that if true parameter $\bbeta$ satisfies $\lambda_0 \|\bbeta\|_\infty > \bbeta' \bC\bbeta/2$ and $\lim\limits_n \lambda_1 /n = \lambda_0 > 0$,
			then \mbox{\boldmath{$\betaSLOPE$}} is not convergent for $\bbeta$.
		\end{remark}
		
		\subsection{Asymptotic pattern recovery in the orthogonal design}\label{sec:Asymptotics in the orthogonal case}
		
		We again consider a sequence of linear models (\eqref{reg2}) but this time we assume that for~each~$n$ the deterministic design matrix $\bX^{(n)}$  of size $n\times p$  satisfies
		\begin{align}\label{cond1}
			(\bX^{(n)})' \bX^{(n)} = n \bI_p.
		\end{align}
		As usual, we assume Gaussian errors, $\beps^{(n)}\sim\mathrm{N}(0,\sigma^2 \bI_n)$. 
		
		\noindent
		Let $\widehat{\bbeta}_n^{SLOPE}=\left(\widehat{\bbeta}^{\slope}_1(n),\ldots,\widehat{\bbeta}^{\slope}_p(n)\right)'$ be the SLOPE estimator defined by \eqref{slope for consistency}. With the above notation we present the main result of this section.
		\begin{theorem}\label{thm:ON}
			Assume that 
			\begin{align*}
				\lim_{n\to\infty}\frac{\lambda_1^{(n)}}{n}=0
			\end{align*}
			and that there exists $\delta>0$ such that
			\begin{align}\label{ass2}
				\underset{{n}\to\infty}{\lim\inf} \frac{ \lambda_i^{(n)}- \lambda_{i+1}^{(n)}}{\sqrt{n}\,(\log(n))^{1/2+\delta}}= m>0
				\quad\mbox{for}\quad i=1,\ldots,p-1.
			\end{align}
			Then we have
			\[
			\pattern(\hat\bbeta^{\slope}_n)\overset{a.s.}{\to}\pattern(\bbeta).
			\]
		\end{theorem} \noindent
		Note that above conditions are satisfied e.g. by $\lambda_{i}^{(n)}=(p+1-i) n^{2/3}$.
		\begin{proof}
			Without loss of generality we may assume that $\bbeta=(\beta_1,\ldots,\beta_p)'$ and $\beta_1\geq \beta_2\geq\ldots\geq\beta_p\geq0$.	Indeed, we can always achieve such condition by permuting the columns of~$\bX^{(n)}$ and~changing their signs.\\
			Since the space of models is discrete, we have to show that for large $n$,\\ $\pattern(\hat\bbeta^{\slope}_n)=\pattern(\bbeta)$ a.s. We divide the proof into the following four parts:
			\begin{enumerate}
				\item[(a)] $\beta_i=\beta_{j}>0$ $\implies$ $\widehat{\beta}^{\slope}_i(n) =\widehat{\beta}^{\slope}_{j}(n)$ a.s. for large $n$, 
				\item[(b)] $\beta_i>\beta_{i+1}$ $\implies$ $\widehat{\beta}^{\slope}_i(n) >\widehat{\beta}^{\slope}_{i+1}(n)$ a.s. for large $n$, 
				\item[(c)] $\beta_i=0$ $\implies$ $\widehat{\beta}^{\slope}_i(n) =0$ a.s.  for large $n$,
				\item[(d)]  $\beta_i>0$ $\implies$ $\widehat{\beta}^{\slope}_i(n) >0$ a.s. for large $n$.
			\end{enumerate} 
			The points (b) and (d) follow quickly by the strong consistency of $\hat\bbeta^{\slope}(n)$.
			To~prove (a) and (c) we reduce the problem to the orthogonal design case. We have
			\begin{align*}
				&\operatorname*{arg\,min}_{{\mathbf b} \in \mathbb{R}^p}& [ \,\frac{1}{2} \|\bY^{(n)}-\bX^{(n)}\bb \|_2^2+ J_{\mathbf \Lambda^{(n)}}(\bb)\,]\\
				= &\operatorname*{arg\,min}_{{\mathbf b} \in \mathbb{R}^p}&   [ \,\frac{1}{2} \|(\widetilde{\bY}^{(n)})'-(\widetilde{\bX}^{(n)})'\bb \|_2^2+ J_{{\widetilde{\mathbf{\Lambda}}}^{(n)}}(\bb)\,  ],
			\end{align*}
			where $\widetilde{\bY}^{(n)} = \bY_n/\sqrt{n}$, $\widetilde{\bX}^{(n)}=\bX^{(n)}/\sqrt{n}$ and $\widetilde{\bLambda}^{(n)} = \bLambda^{(n)}/n$. Clearly, \eqref{cond1} implies that $(\widetilde{\bX}^{(n)})' \widetilde{\bX}^{(n)}=\bI_p$, which allows to use results from the orthogonal design. However, we~note that the OLS estimators $\widehat{\bbeta}_n^{\ols}=(\widehat{\bbeta}^{\ols}_{1}(n),\ldots,\widehat{\bbeta}^{\ols}_{p}(n))$ are the same in the original model  and its scaled version $\widetilde{\bY}^{(n)}=\widetilde{\bX}^{(n)}\bbeta+\beps^{(n)}/\sqrt{n}$.
			
			Let $\pi_n$ be a permutation of $(1,2,\ldots,p)$ satisfying  
			$$|\widehat{\beta}^{\ols}_{\pi_n(1)}(n)|\geq |\widehat{\beta}^{\ols}_{\pi_n(2)}(n)|\geq \ldots\geq |\widehat{\beta}^{\ols}_{\pi_n(p)}(n)|.$$ 
			By the strong consistency of the OLS estimator, taking $n$ sufficiently large, we may ensure that the clusters of $\bbeta$ do not interlace in $\widehat{\bbeta}_n^{\ols}$ in the sense that if $\beta_i>\beta_j$, then $\widehat{\beta}^{\ols}_{i}(n)>\widehat{\beta}^{\ols}_{j}(n)$ a.s. for $n$ sufficiently large. 
			
			Let us now consider point (a). 
			Let $S_i$ denote the cluster containing $\beta_i>0$, that is, the set $S_i=\{j\in \{1,\ldots,p\}\colon \beta_j=\beta_i\}$. 
			In view of the ordering of $\bbeta$, there exists $k_i\in\{1,\ldots,p\}$ such that 
			\[
			S_i = \left\{ \pi_n(j)\colon j\in\{k_i, k_i+1,\ldots,k_i+\#S_i-1 \}\right\}.
			\]
			We will show that if $\pi_n(k), \pi_n(k+1)\in S_i$, then for large $n$
			\begin{align}\label{eq:eq}
				\widehat{\beta}^{\slope}_{\pi_n(k)}(n)=\widehat{\beta}^{\slope}_{\pi_n(k+1)}(n) \quad\mbox{a.s.},
			\end{align}
			thus $\widehat{\beta}^{\slope}_{j}(n)=\widehat{\beta}^{\slope}_{k}(n)$ for $j,k\in S_i$, which finishes the proof of (a).\\
			Now assume that $\pi_n(k), \pi_n(k+1)\in S_i$. Then, by  Theorem~\ref{FAKT 2}, the condition \eqref{eq:eq} is satisfied if 
			\begin{align}\label{eq:toShow}
				\lt|\widehat{\beta}^{\ols}_{\pi_n(k)}(n)\rt| - \lt|\widehat{\beta}^{\ols}_{\pi_n(k+1)}(n)\rt| \leq  \widetilde{\Lambda}_k^{(n)}-\widetilde{\Lambda}_{k+1}^{(n)} = \frac{1}{n}\left(\lambda_k^{(n)}-\lambda_{k+1}^{(n)}\right)
			\end{align}
			holds for large $n$ and both $\widehat{\beta}^{\ols}_{\pi_n(k)}(n)$ and $\widehat{\beta}^{\ols}_{\pi_n(k)}(n)$ have the same sign. The latter is ensured by the strong consistency of the OLS estimator and the fact that $\beta_i>0$. 
			
			If $\pi_n(k),\pi_{n}(k+1)\in S_i$, then we have the following bound
			\begin{align}\label{eq:ineq}
				\left|\widehat{\beta}^{\ols}_{\pi_n(k)}(n) - \widehat{\beta}^{\ols}_{\pi_n(k+1)}(n)\right| \leq \sum_{j\in S_i} \left|\widehat{\beta}^{\ols}_{j}(n) - \widehat{\beta}^{\ols}_{i}(n)\right|.
			\end{align}
			Take any $j\in S_i$. Since both $\widehat{\beta}^{\ols}_{j}(n)$ and $\widehat{\beta}^{\ols}_{i}(n)$ have the normal distribution with the~same mean, by Lemma \ref{lem:gauss}, we have
			\[
			\lim_{n\to\infty}\frac{\sqrt{n}\left(\widehat{\beta}^{\ols}_{j}(n) - \widehat{\beta}^{\ols}_{i}(n)\right)} {(\log(n))^{1/2+\delta}}= 0 \quad\mbox{a.s.}
			\]
			In view of \eqref{eq:ineq} and \eqref{ass2}, this implies that \eqref{eq:toShow} holds true for large $n$. Hence, (a)~follows.\\
			It remains to establish (c). Assume that $\beta_{p_0}>0=\beta_{p_0+1}=\ldots=\beta_p$. Clearly, condition (a) from Theorem \ref{supportON} is satisfied thanks to the strong consistency of the OLS estimator.  For (b),
			we have for $k=1,2,\ldots,p_0$,
			\[
			\sum_{i=k}^{p_0}\widetilde{\Lambda}_i^{(n)} = \frac{1}{n} \sum_{i=k}^{p_0}\lambda_i^{(n)}\leq p_0 \frac{\lambda_1^{(n)}}{n},
			\]
			which converges to $0$. On the other hand, the left-hand side of (b) converges a.s. to $\sum_{i=k}^{p_0} \beta_i$, which is positive. Thus, condition (b) from Theorem \ref{supportON} holds for large~$n$. Condition (c) from Theorem \ref{supportON} follows from Lemma \ref{lem:gauss}. Indeed, we have for $\delta>0$ and $k=p_0+1,\ldots,p$,
			\[
			\lim_{n\to\infty} \frac{\sqrt{n}}{(\log(n))^{1/2+\delta}} \sum_{i=p_0+1}^k |\widehat{\beta}^{\ols}_{\pi_n(i)}(n)| = \sum_{i=p_0+1}^k  \lim_{n\to\infty} \frac{|\sqrt{n}\,\widehat{\beta}^{\ols}_{\pi_n(i)}(n)| }{(\log(n))^{1/2+\delta}} =0 \mbox{ a.s.},
			\]
			while 
			\[
			\lim_{n\to\infty} \frac{\sqrt{n}}{(\log(n))^{1/2+\delta}}   \sum_{i=p_0+1}^k \widetilde{\Lambda}_i^{(n)} \geq  \sum_{i=p_0+1}^k  \lim_{n\to\infty} \frac{\lambda_i^{(n)}-\lambda_{i+1}^{(n)}}{\sqrt{n}(\log(n))^{1/2+\delta}}  = m >0
			\]
			Thus, all assumptions of Theorem \ref{supportON} are verified and the proof is complete.
		\end{proof}

		\section{Applications and simulations}\label{sec:appsim}
		Below we present an application of SLOPE in signal denoising. In our example ${\bX\in\R^{300\times 100}}$ is an orthogonal system of trigonometric functions, i.e.\\ $X_{i,(2*j-1)}=\sin(2\pi ij/n)$ and~$X_{i,(2*j)}=\cos(2\pi ij/n)$ for $i=1,\ldots,100$\\ and~${j=1,\ldots,150}$. Here $\bbeta\in\R^p$ is a vector consisting~of two~clusters: $20$ coordinates with absolute value $100$ and $20$ coordinates with absolute value $80$. The absolute values of coordinates of $\bbeta$ are sorted in a decreasing way. The signs of the nonzero coordinates are chosen independently with random uniform distribution. To avoid large bias caused by the shrinkage nature of LASSO and SLOPE, we debias them by combining with the OLS method. For that reason we use the following definition of the pattern matrix $\bU_M$ and the~clustered design matrix $\widetilde{\bX}_M$, which is based on the SLOPE pattern:
		\begin{definition}\label{SLOPE pattern matrix}
			Let $M\neq 0$ be a  pattern  in $\mathcal{M}_p$ with  $k=\|M\|_\infty$ nonzero clusters.
			The pattern matrix $\bU_M\in \R^{p\times k}$ is defined as follows
			\[
			(\bU_M)_{ij}=\sign(m_i){\bf 1}_{(|m_i|=k+1-j)},\qquad i \in \{1,\dots,p\}, \,j \in \{1,\dots,k\}.
			\]
		\end{definition}
		\begin{definition}\label{def:3}
			Let $M\neq 0$ be a pattern in $\R^p$ and $k=\max\{\|M\|_\infty,1\}$. For $\bX\in\R^{n\times p}$ we define the clustered design matrix by $\tilde \bX_M=\bX \bU_M\in\R^{n\times k}$.
		\end{definition}
		
		To perform the debiased SLOPE, we begin with recovering the support and clusters of a~true vector $\bbeta$ with SLOPE. Then, using the obtained SLOPE pattern $M$, we replace the design matrix with its clustered version $\widetilde{\bX}_M = \bX\mathbf{U}_M$. 
		Then we perform the Ordinary Least Squares regression for the model $\bY = \widetilde{\bX}_M \mathbf{b} + \beps$, where $\mathbf{b}$ consists only of distinct absolute values of \mbox{\boldmath{$\betaSLOPE$}}.
		
		
		Analogously we proceed with the debiased LASSO. However, in this method we use the~LASSO pattern matrix defined in a following way:
		
		For LASSO we have the LASSO pattern being a sign vector, cf. \cite{geom2}. For $\bS\in\{-1,0,1\}^p$, $\|\bS\|_1$ denotes the number of nonzero coordinates. If $\|bS\|_1=k\geq1$, then we define the~corresponding pattern matrix $\bU_S\in \R^{p\times k}$ by 
		\[
		\bU_S= \mathrm{diag}(\bS)_{\mathrm{supp}(\bS)},
		\]
		i.e. the submatrix of $\mathrm{diag}(\bS)$ obtained by keeping columns corresponding to indices in~$\mathrm{supp}(\bS)$.  
		Then we define the reduced matrix $\tilde{\bX}_S$
		by 
		\[
		\tilde{\bX}_S=\bX \bU_S.
		\]
		Equivalently, we have $\tilde{\bX}_S = (S_i X_i)_{i\in\mathrm{supp}(S)}$. The notion of pattern matrix also appears in~\cite{modrec}.
		In our example $\beps\in\mathcal{N}(0,\sigma^2 \bI_n)$ and $\sigma=30$. 
		
		We compare the Mean Square Error and the signal denoising of the classical OLS estimation, the LASSO with the tuning parameter $\lambda_{cv}$ minimizing the cross-validated error, the denoised version of LASSO with $\lambda=5\lambda_{cv}$ and the denoised version of~SLOPE with the tuning vector $\bLambda$ chosen with respect to the scaled arithmetic sequence ($\lambda_i = 3.5(p+1-i)$). 
		\begin{figure*}[ht]
			\centering
			\begin{subfigure}{.4\textwidth}
				\centering
				\includegraphics[scale=0.25]{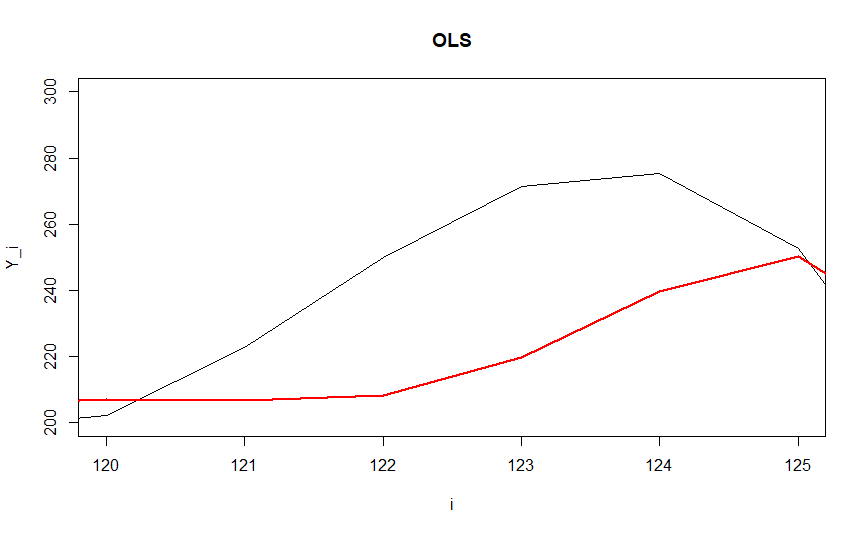}
			\end{subfigure}
			\hfill
			\begin{subfigure}{.4\textwidth}
				\centering
				\includegraphics[scale=0.25]{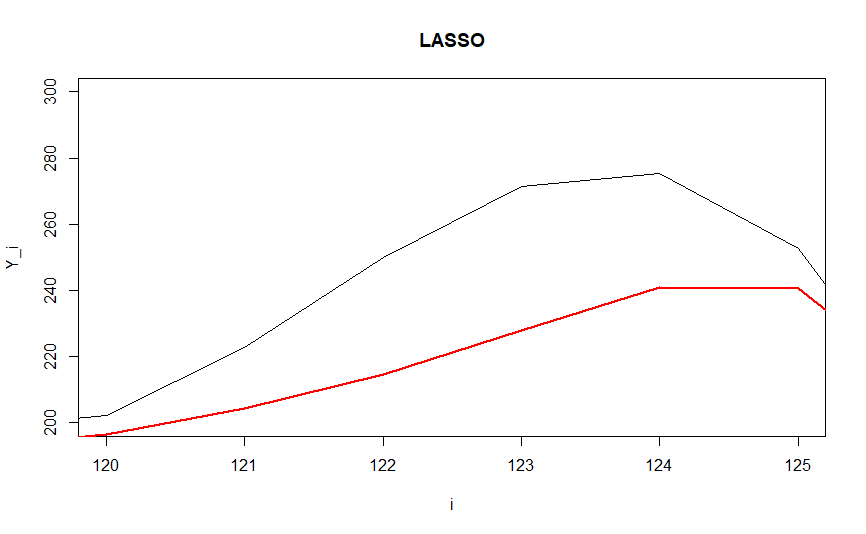}
			\end{subfigure}
		\end{figure*}
		\setcounter{figure}{1}
		\begin{figure}[ht]
			\begin{subfigure}{.4\textwidth}
				\centering
				\includegraphics[scale=0.25]{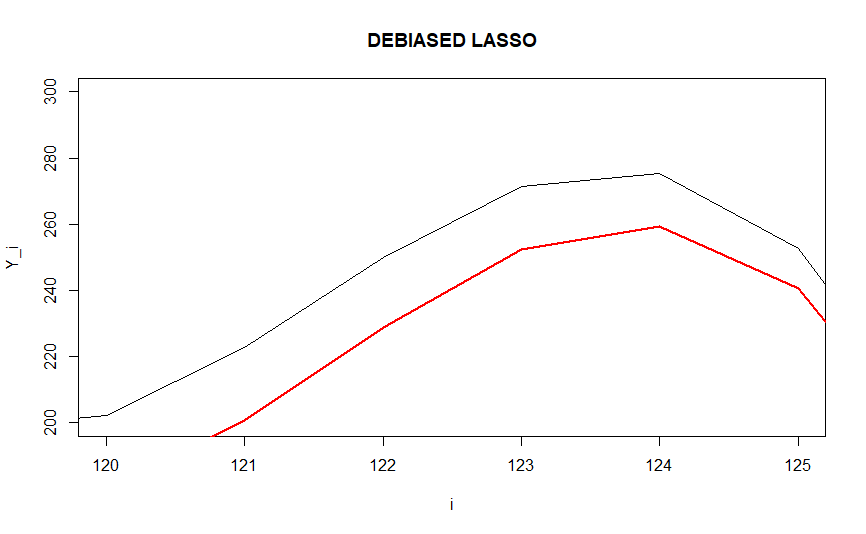}
			\end{subfigure}
			\hfill
			\begin{subfigure}{.4\textwidth}
				\centering
				\includegraphics[scale=0.25]{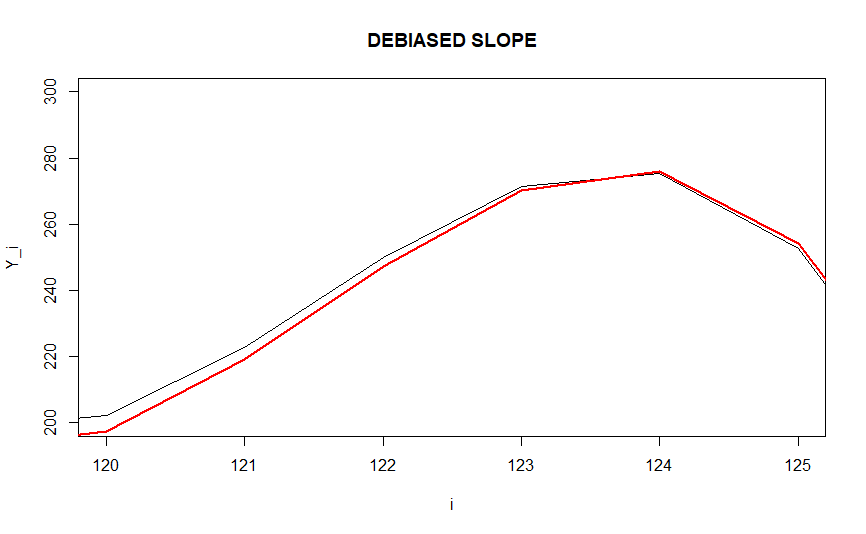}
			\end{subfigure}
			\caption{Comparison of signal denoising by OLS, LASSO, debiased LASSO and debiased SLOPE (respective images from left) on the coordinates $[120,125]$ of the regression model $\bY=\bX\bbeta+\beps$. The black lines correspond to the true values of $\bX\bbeta$. The red lines correspond to the estimators $\bY=\bX\hat\bbeta$.}
		\end{figure}
		
		\begin{figure}
			\centering
			\includegraphics[scale=0.35]{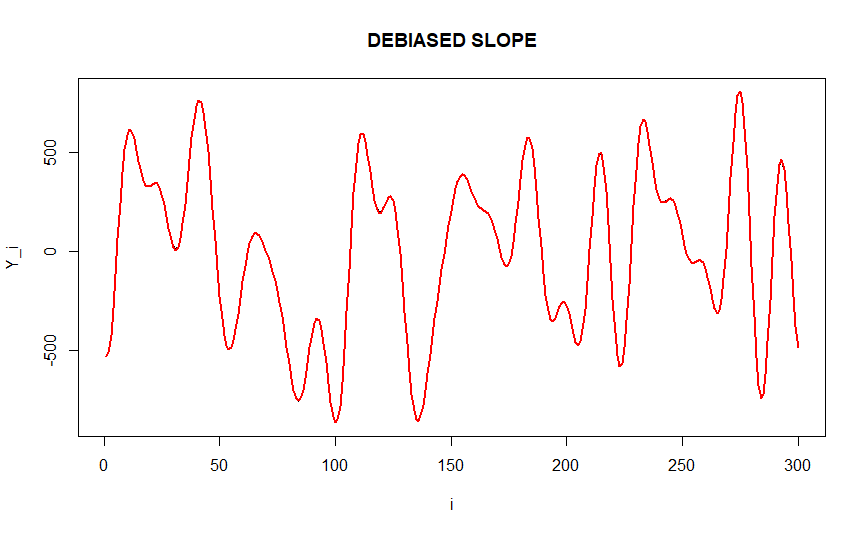}
			\caption{Signal denoising by debiased SLOPE on all coordinates of the regression model $\bY=\bX\bbeta+\beps$. The (almost overlapping) black line and the red line correspond respectively to the true values of $\bX\bbeta$ and to $\bY=\bX\hat\bbeta^{SLOPE}$.}
		\end{figure}
		
		\begin{figure}
			\begin{subfigure}{.4\textwidth}
				\centering
				\includegraphics[scale=0.3]{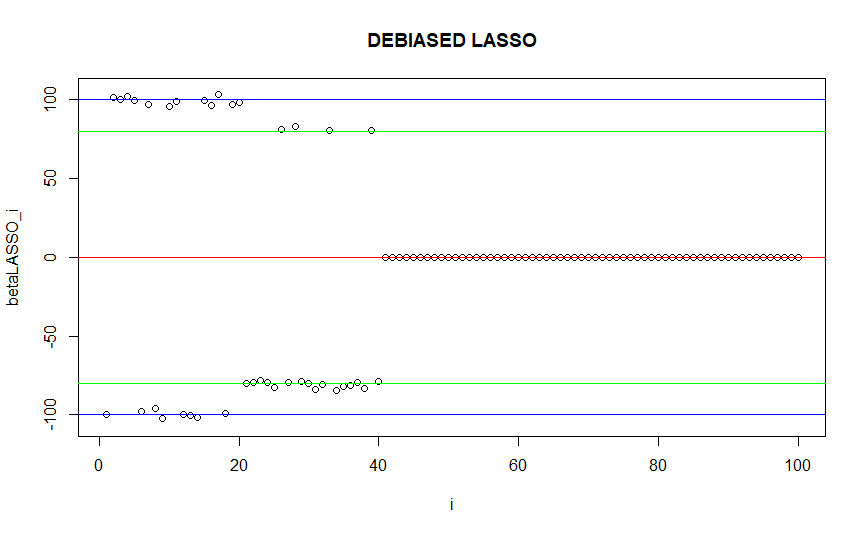}
			\end{subfigure}
			\hfill
			\begin{subfigure}{.4\textwidth}
				\centering
				\includegraphics[scale=0.3]{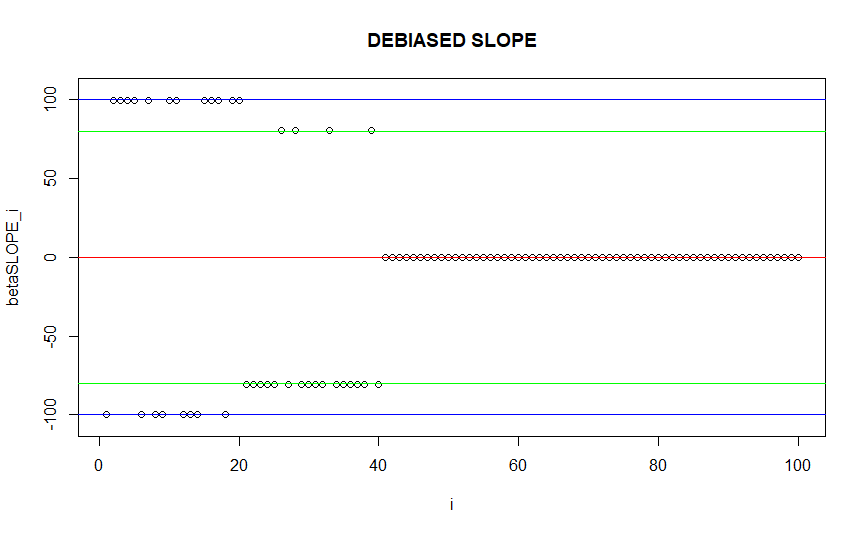}
			\end{subfigure}
			\caption{Pattern recovery by debiased LASSO (left image) and by debiased SLOPE (right image) in the same setting as above.}
			\label{fig:vs}
		\end{figure}
		
		We also compare debiased SLOPE with debiased LASSO, as shown in Figure~\ref{fig:vs}. The horizontal lines correspond to the true values of $\bbeta$. As~one~may~observe, in the presented setting LASSO does not recover the true support, while debiased SLOPE perfectly recovers support, sign and clusters.  
		
		\begin{table}
			\centering
			\begin{tabular}{|c|c|c|c|c|}\hline
				& OLS & LASSO-CV & LASSO-LS & SLOPE-LS \\ \hline
				$MSE(\bbeta,\cdot)$ & 613.6797 & 426.3705 & 171.7957 & 20.74967\\ \hline
			\end{tabular}
			\caption{Comparison of MSE between different regression methods.}
			\label{tab:1}
		\end{table}
		\appendix
		\section{Appendix}\label{appendix}
		\begin{proof}[Proof of Lemma {\ref{lemma 2 bis}}] Since the matrix $(\bX'\bX)^{-1}$ is nonnegative definite, it follows that the function $g :C_{\mathbf \Lambda} \rightarrow [0,\infty)$   defined by
			$$
			g(\bpi):= (\bX'\bY - \bpi)'(\bX'\bX)^{-1}(\bX'\bY - \bpi)
			$$
			is convex in $\bpi$. Therefore, at the point $\bpi^*=(\pi_1^*,\ldots,\pi_p^*)'$, where $g$ attains its~global minimum over $C_{\mathbf \Lambda}$, the gradient $\nabla g$ of $g$ satisfies
			\begin{equation*} \label{gradient inequality} \nonumber
				\left[\,\nabla g(\bpi^*) \,\right]'(\bpi-\bpi^*) \ge 0, \quad \mbox{\rm for all } \; \bpi \in C_{\mathbf \Lambda}.
			\end{equation*}
			This implies $(\bpi-\bpi^*)' \bbeta^*   \leq  0$,  for all  $\bpi \in C_{\mathbf \Lambda}$,  because 
			$$
			\nabla g(\bpi^*)=-2(\bX'\bX)^{-1}(\bX'\bY - \bpi^*)=-2\bbeta^*.
			$$
			In the proof of parts (a), (b) and (c) we use the fact that $\bpi^*$ maximizes $\bpi' \bbeta^*$ over $\bpi \in C_{\mathbf \Lambda}$.\\
			To prove part (a)  suppose  that $\mbox{\rm sign}(\beta_i^*) \cdot \mbox{\rm sign}(\pi_i^*) < 0$ for some $i$ and define $$\bpi=(\pi_1^*,\ldots,\pi_{i-1}^*,-\pi_{i}^*,\pi_{i+1}^*,\ldots,\pi_{p}^* )'.$$ Then we have $(\bpi^*)' \bbeta^*   <  \bpi' \bbeta^*$,   
			which is impossible since $ \bpi \in C_{\mathbf \Lambda}$.\\  
			To prove part (b), consider a permutation $\tau$ of $(1,2,\ldots,n)$ such that\\ $(|\pi_{\tau(1)}^*|,\ldots,|\pi_{\tau(p)}^*|)$ and $(|\beta_1^*|,\ldots,|\beta_p^*|)$ are similarly sorted. Define the point\\ 
			$\bpi=(s_1 \cdot \pi^*_{\tau(1)}, s_2 \cdot \pi^*_{\tau(2)},\ldots,s_p \cdot \pi^*_{\tau(p)})$, where $s_i=\mbox{\rm sign}(\beta_i^*)$, for $i=1,2,\ldots,p$. If  $(|\pi_{\tau(1)}^*|,\ldots,|\pi_{\tau(p)}^*|) \neq (|\pi_1^*|,\ldots,|\pi_p^*|)$, then, by~the~Hardy-Littlewood-P\'{o}lya rearrangement inequality,
			$$
			\bpi' \bbeta^*=\sum_{i=1}^p |\pi^*_{\tau(i)}||\bbeta_i^*|>  \sum_{i=1}^p |\pi^*_i||\bbeta_i^*| \ge  (\bpi^*)' \bbeta^*,
			$$
			which is impossible since $\bpi \in C_{\mathbf \Lambda}$.
			\\
			Finally, to prove part (c), suppose that $\sum_{i=1}^{k-1} |\pi^*_{\tau(i)}|< \sum_{i=1}^{k-1} \lambda_i$,  and that  $|\pi^*_{\tau(k)}|>0$. In~this case there is a sufficiently small $\delta>0$, such that
			$$ 
			\bpi=(\pi_1^*,\ldots,\pi_{i-2}^*,\pi_{i-1}^*+\delta s_{i-1},\pi_{i}^*-\delta s_i,\pi_{i+1}^*,\ldots,\pi_{p}^* )' \in C_{\mathbf \Lambda}.
			$$
			If $|\beta^*_{\tau(k-1)}|>|\beta^*_{\tau(k)}|$  then
			$$
			\bpi' \bbeta^*=(\bpi^*)' \bbeta^*+\delta(|\beta^*_{\tau(k-1)}|-|\beta^*_{\tau(k)}|)> (\bpi^*)' \bbeta^*,
			$$
			which is impossible.
		\end{proof}
		
		\begin{proof}[Proof of Lemma {\ref{lemma 3}}]  At first we note that for all $\bpi \in C_{\mathbf \Lambda}$
			\begin{eqnarray*}
				r(\bbeta^*,\bpi) &=& \frac{1}{2}\|\bY-\bX\bbeta^* \|_2^2+\bpi' \bbeta^*=\frac{1}{2}\|\bY-\bX\bbeta^* \|_2^2+ (\bpi^*)' \bbeta^* \\
				& & + (\bpi -\bpi^*) ' \bbeta^*= r(\bbeta^*,\bpi^*)+(\bpi -\bpi^*) ' \bbeta^* \le r(\bbeta^*,\bpi^*), 
			\end{eqnarray*}
			where the last inequality follows from the fact that $(\bpi-\bpi^*)' \bbeta^*   \leq  0$,  for all  $\bpi~\in~C_{\mathbf \Lambda}$, see the proof of~\ref{lemma 2 bis}. 
			Therefore, $\displaystyle 
			\max_{{\mathbf \pi} \in C_{\mathbf \Lambda}} \; r(\bbeta^*,\bpi)=r(\bbeta^*,\bpi^*)$. Moreover, from the definition of~the~point $\bbeta^*$ it is seen that $\displaystyle r(\bbeta^*,\bpi^*) =\min_{{\mathbf \beta} \in {\cal M}}\; r(\bbeta,\bpi^*) $. These two facts imply that
			\begin{eqnarray*}
				\min_{{\mathbf \beta} \in {\cal M}} \max_{{\mathbf \pi} \in C_{\mathbf \Lambda}} \; r(\bbeta,\bpi) &\leq & \max_{{\mathbf \pi} \in C_{\mathbf \Lambda}} \; r(\bbeta^*,\bpi) = r(\bbeta^*,\bpi^*)\\ & = &\min_{{\mathbf \beta} \in {\cal M}}\; r(\bbeta,\bpi^*) \leq \max_{{\mathbf \pi} \in C_{\mathbf \Lambda}} \min_{{\mathbf \beta} \in {\cal M}}\; r(\bbeta,\bpi).
			\end{eqnarray*} 
			Since  $\displaystyle \max_{{\mathbf \pi} \in C_{\mathbf \Lambda}} \min_{{\mathbf \beta} \in {\cal M}}\; r(\bbeta,\bpi) \leq \min_{{\mathbf \beta} \in {\cal M}} \max_{{\mathbf \pi} \in C_{\mathbf \Lambda}} \; r(\bbeta,\bpi)$ (by the max-min inequality),  we~have the~equality throughout.
			This completes the proof. 
		\end{proof}
		
		\begin{proof}[Proof of Lemma {\ref{lemat 1}}] 
			Observe that
			\begin{eqnarray*}
				\| \bY - \bX\bb \|_2^{2}&=& \bY'\bY - 2\bY'\bX\bb + \bb'\bb\\
				\|\widehat{\bbeta}^{\ols}-\bb\|_2^{2}&=& \bY'\bX\bX'\bY - 2\bY'\bX\bb + \bb'\bb.
			\end{eqnarray*}
			Therefore both optimization problems differ by $\frac{1}{2}(\bY'\bY - \bY'\bX\bX'\bY)$, which does not depend on $\bb$, which implies their equivalence.
		\end{proof}
		

		
		\bibliographystyle{plain}
		
			\end{document}